\documentclass[a4paper,11pt]{amsart}

\usepackage{amsmath,amssymb,amsthm}
\usepackage{a4wide}
\usepackage{hyperref}
\usepackage{array}
\usepackage{color}

\usepackage{graphicx}
\usepackage{xypic}
\entrymodifiers={+!!<0pt,\fontdimen22\textfont2>}
\usepackage[all]{xy}

\newtheoremstyle{myremark} 
    {7pt}                    
    {7pt}                    
    {}  	                 
    {}                           
    {\bf}       	         
    {.}                          
    {.5em}                       
    {}  

\newtheorem{innercustomthm}{Theorem}
\newenvironment{customthm}[1]
  {\renewcommand\theinnercustomthm{#1}\innercustomthm}
  {\endinnercustomthm}

\newtheorem{innercustomclaim}{Claim}
\newenvironment{customclaim}[1]
  {\renewcommand\theinnercustomclaim{#1}\innercustomclaim}
  {\endinnercustomclaim}

\theoremstyle{plain}
\newtheorem{lemma}{Lemma}[section]
\newtheorem{theorem}[lemma]{Theorem}

\newtheorem{definition}[lemma]{Definition}

\newtheorem{proposition}[lemma]{Proposition}
\newtheorem{conjecture}[lemma]{Conjecture}

\newtheorem{problem}[lemma]{Problem}
\newtheorem{notation}[lemma]{Notation}
\newtheorem*{claim}{Claim}

\theoremstyle{myremark}
\newtheorem{remark}[lemma]{Remark}
\newtheorem{example}[lemma]{Example}

\newcommand{\RR}{\mathbb{R}}

\newcommand{\ZZ}{\mathbb{Z}}
\newcommand{\CC}{\mathbb{C}}
\newcommand{\nices}{\mathcal{S}}

\newcommand{\nicea}{\mathcal{A}}
\newcommand{\niceb}{\mathcal{B}}
\newcommand{\niceo}{\mathcal{O}}

\newcommand{\htpyequiv}{\simeq}

\newcommand{\incl}{\hookrightarrow}
\renewcommand{\subset}{\subseteq}

\newcommand{\lk}{\mathrm{lk}}

\newcommand{\ost}{\mathrm{st}^o}

\newcommand{\diam}{\mathrm{diam}}
\newcommand{\rips}{\mathcal{R}}
\newcommand{\shadow}{\mathcal{S}}

\DeclareMathOperator{\conv}{\mathrm{conv}}

\begin{document}
\title[On homotopy types of Euclidean Rips complexes]{On homotopy types of Euclidean Rips complexes}
\author{Micha{\l} Adamaszek}
\address{Department of Mathematical Sciences, University of Copenhagen, Universitetsparken 5, 2100 Copenhagen, Denmark}
\email{aszek@mimuw.edu.pl}
\author{Florian Frick}
\address{Department of Mathematics, Cornell University, Ithaca, NY 14853, USA}
\email{ff238@cornell.edu}
\author{Adrien Vakili}
\address{Department of Mathematical Sciences, University of Copenhagen, Universitetsparken 5, 2100 Copenhagen, Denmark}
\email{adrienvakili@gmail.com}

\thanks{MA supported by VILLUM FONDEN through the network for Experimental Mathematics in Number
Theory, Operator Algebras, and Topology.}

\begin{abstract}
The Rips complex at scale $r$ of a set of points $X$ in a metric space is the abstract simplicial complex whose faces are determined by finite subsets of $X$ of diameter less than~$r$. We prove that for $X$ in the Euclidean $3$-space $\RR^3$ the natural projection map from the Rips complex of $X$ to its shadow in $\RR^3$ induces a surjection on fundamental groups. This partially answers a question of Chambers, de Silva, Erickson and Ghrist who studied this projection for subsets of~$\RR^2$. We further show that Rips complexes of finite subsets of $\RR^n$ are universal, in that they model all homotopy types of simplicial complexes PL-embeddable in~$\RR^n$. As an application we get that any finitely presented group appears as the fundamental group of a Rips complex of a finite subset of~$\RR^4$. We furthermore show that if the Rips complex of a finite   point set in $\RR^2$ is a normal pseudomanifold of dimension at least two then it must be the boundary of a crosspolytope.
\end{abstract}
\maketitle

\section{Introduction}
\label{sect:intro}

The Rips (or Vietoris--Rips) complex is an abstract simplicial complex which records the notion of pairwise proximity between the points of a metric space. First defined by Vietoris~\cite{Vietoris27}, and heavily used in geometric group theory, it has recently become one of the tools in computational algebraic topology, especially in the framework of persistent homology. There the main utility of the Rips complex, and other constructions such as the \v{C}ech complex, is in building continuous approximations to discrete point-cloud datasets, making it possible to study them with the toolbox of algebraic topology. In this spirit we restrict attention to the spaces $\RR^n$ equipped with the Euclidean metric. We denote the distance between two points $A,B$ by $|AB|$, the diameter of a set $Y\subseteq \RR^n$ by $\diam(Y)$ and its convex hull by $\conv(Y)$.

\begin{definition}[Rips complex]
The Rips complex $\rips(X;r)$ of a subset $X\subseteq \RR^n$ at distance scale $r>0$ is the abstract simplicial complex with vertex set $X$, such that a finite set $Y\subseteq X$ is a face of $\rips(X;r)$ if and only if $\diam(Y)< r$.
\end{definition}

Alternatively, one can define $\rips(X;r)$ as the clique (flag) complex of the graph with vertex set $X$ and edges $A\sim B$ whenever $|AB|<r$. It is easy to see that for $X\subseteq \RR^1$ each connected component of $\rips(X;r)$ is contractible, but surprisingly little is known about the homotopical properties of Rips complexes of subsets of $\RR^n$, $n\geq 2$. Motivated by applications in sensor networks, Chambers, de Silva, Erickson and Ghrist~\cite{chambers2010vietoris} studied Rips complexes of subsets of~$\RR^2$. They compare the fundamental groups of Rips complexes in $\RR^2$ and their \emph{shadows}, which we now define.

\begin{definition}[Shadow, projection map]
Suppose $X\subseteq \RR^n$. The shadow $\shadow(X;r)\subseteq\RR^n$ is the image of $\rips(X;r)$ under the projection map $p\colon\rips(X;r)\to\RR^n$ which sends each vertex to the corresponding point in $X$ and extends linearly to the simplices of $\rips(X;r)$. Equivalently, we have
\begin{equation}
\label{eq:shadow-def}
\shadow(X;r)=\bigcup_{\substack{Y\subset X,\ |Y|\leq n+1\\\diam(Y)< r}}\conv(Y).
\end{equation}
\end{definition}
\begin{figure}
\label{fig:rips-shadow}
\begin{tabular}{m{1.5in}m{0.2in}m{1.5in}}
\includegraphics[scale=0.8]{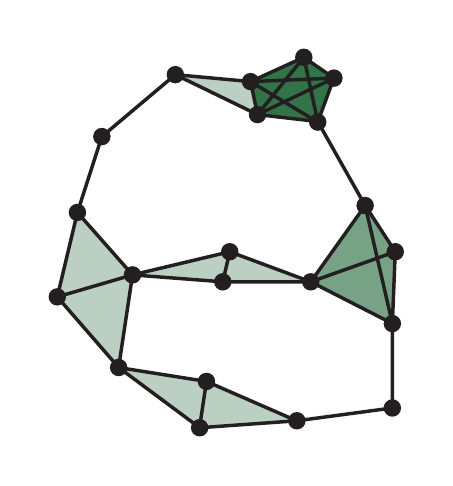}&  $\xrightarrow{\ p\ }$ &\includegraphics[scale=0.8]{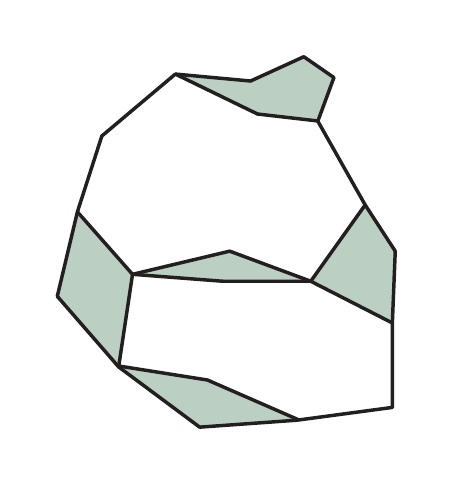} 
\end{tabular}
\caption{
A point set in $\RR^2$ (black dots), its Rips complex and the projection map from the Rips complex to the shadow.}
\end{figure}
See Figure~\ref{fig:rips-shadow} for an example. The restriction in \eqref{eq:shadow-def} to subsets with $|Y|\leq n+1$ is possible thanks to Carath\'e{}odory's theorem, which states that the convex hull of a finite set in $\RR^n$ is the union of convex hulls of its at most $(n+1)$-element subsets.

Since most of our discussion is scale-invariant, we usually set $r=1$ and write $\rips(X):=\rips(X;1)$ and $\shadow(X):=\shadow(X;1)$. The results of \cite{chambers2010vietoris} can be summarized as follows.

\begin{theorem}[\cite{chambers2010vietoris}]
\label{thm:chambers}
Let $X\subseteq \RR^n$ be a finite subset. The projection map $p\colon\rips(X)\to \shadow(X)$ 
\begin{itemize}
\item[(a)] is a $\pi_0$-isomorphism for all $n$,
\item[(b)] is a $\pi_1$-isomorphism when $n=2$,
\item[(c)] may fail to be $\pi_1$-surjective for each $n\geq4$.
\end{itemize}
\end{theorem}
The main result of \cite{chambers2010vietoris} is of course part (b), the $\pi_1$-isomorphism in the planar case. Our first theorem generalizes half of this statement to dimension three.

\begin{customthm}{I}
\label{theoremA}
For any finite subset $X\subseteq \RR^3$ the projection map $p\colon\rips(X)\to \shadow(X)$ is a $\pi_1$-surjection.
\end{customthm}
Because of part (c), this was the borderline case for $\pi_1$-surjectivity of $p$. The proof of Theorem~\ref{theoremA} appears in Section~\ref{sect:local}. We review the necessary homotopy theory prerequisites in Section~\ref{sect:topology} and prepare some $3$-dimensional geometry in Section~\ref{sect:geometry}.  The proof is inductive. It isolates the local geometric properties of shadows from a more generic algebro-topological framework and we hope that it  
can be pushed further and used in other circumstances.

\smallskip
Part (b) of Theorem~\ref{thm:chambers} implies that for $X\subseteq \RR^2$ the group $\pi_1(\rips(X))$ is free. To our best knowledge this is the only known non-trivial piece of information about the homotopy types of Rips complexes of subsets of Euclidean spaces. In contrast, we will show that Rips complexes of subsets of $\RR^n$ are rich enough to model all homotopy types of finite polyhedra in the same~$\RR^n$.

\begin{customthm}{II}
\label{theoremB}
For every finite simplicial complex $K$ that is PL-embeddable in $\RR^n$ there exists a finite subset $X\subseteq \RR^n$ with a homotopy equivalence $\rips(X)\htpyequiv K$. 
\end{customthm}

This result is shown in Section~\ref{sect:universal}, using a combination of the Nerve Lemma with the techniques due to Hausmann~\cite{Hausmann1995} and Latschev~\cite{Latschev2001}. It will immediately imply the following corollary, which we prove also in Section~\ref{sect:universal}.

\begin{customthm}{III}
\label{theoremC}
For every finitely presented group $G$ there exists a finite subset $X\subseteq \RR^4$ such that $\pi_1(\rips(X))=G$.
\end{customthm}

It is likely that also in dimension three the projection $p$ is a $\pi_1$-isomorphism (see Section~\ref{sect:closing}). That would strengthen Theorem~\ref{theoremA}, fully answering the question of~\cite{chambers2010vietoris}. It would also place extra restrictions on $\rips(X)$ for $X\subseteq \RR^3$, such as torsion-freeness of~$H_1(\rips(X);\ZZ)$. At the moment these questions about dimension three remain open.

\smallskip
Going back to dimension two, we address the question raised in \cite[6.(1)]{chambers2010vietoris}, of whether  the only significant examples of homology in dimensions $d\geq 2$ of a planar Rips complex can be generated by ``local'' point configurations. Although we are not able to resolve this problem completely, we show that this is the case for homology generators corresponding to induced normal pseudomanifolds in a planar Rips complex. We will prove the following theorem in Section~\ref{sect:submanifold}.

\begin{customthm}{IV}
\label{theoremD}
Suppose that $X\subset \RR^2$ is a finite subset, $d\geq 2$ and that $\rips(X)$ contains a $d$-dimensional  normal pseudomanifold $K$ as an induced subcomplex. Then $K$ is isomorphic to the boundary of the $(d+1)$-crosspolytope.
\end{customthm}

\begin{remark}
\label{rem:nonstrict}
The Rips complex can also be defined using a non-strict inequality ${\diam(Y)\leq r}$ (this definition was used in \cite{chambers2010vietoris}). If $X$ is finite then the non-strict version can be treated as a strict one by slightly increasing the distance scale and vice-versa. Thus all the results of this paper concerning finite sets $X$ and all results of \cite{chambers2010vietoris} hold in both settings.
Only in Lemma~\ref{lem:crushing} and Theorem~\ref{thm:infinite} we deal with Rips complexes for infinite $X$, and we rely on a theorem of Hausmann~\cite{Hausmann1995} which holds only in the strict setting. In general, for infinite sets $X\subseteq\RR^n$, the strict Rips complexes tend to be better behaved than the non-strict ones; for instance the latter can have homology groups of uncountable rank while the former cannot~\cite[Section~5.2]{chazal2014persistence}.
\end{remark}

\subsection*{Acknowledgements}
We thank Jesper M. M\o{}ller for helpful discussions and for suggesting the collaboration of the first and third author.
Some of this research was performed while the second author visited the University of Copenhagen. The second
author is grateful for the hospitality of the Department of Mathematical Sciences there.

\section{Topological preliminaries}
\label{sect:topology}
We follow standard simplicial terminology, see for instance~\cite{Kozlov}, without distinguishing between an abstract simplicial complex and its geometric realization.

\begin{notation}
\label{notation:decomp}
For $x\in \RR^n$ we denote by $B(x,r)$ the open ball around $x$ of radius~$r$.

For a subset $X\subseteq\RR^n$ and $v\in X$ we define \begin{eqnarray*}
X_v&=&(X-v)\cap B(v,1),\\
X^v&=&X \cap B(v,1)\ =\ X_v\cup\{v\}.
\end{eqnarray*}
\end{notation}
This notation is chosen so that $\rips(X_v)$ is the link and $\rips(X^v)$ is the star of $v$ in $\rips(X)$. We have $X=X^v\cup (X-v)$ and $X^v\cap (X-v)=X_v$. 

\begin{lemma}
\label{lem:stars}
For any $X\subseteq \RR^n$ and $v\in X$ the spaces $\rips(X^v)$ and $\shadow(X^v)$ are contractible.
\end{lemma}
\begin{proof}
The simplicial complex $\rips(X^v)$ is a cone with apex $v$. The set $\shadow(X^v)$ is a star-shaped subset of $\RR^n$ with center~$v$.
\end{proof}

If $X,Y\subseteq \RR^n$ then $\rips(X\cap Y)=\rips(X)\cap\rips(Y)$. Shadows always satisfy the inclusion $\shadow(X\cap Y)\subseteq \shadow(X)\cap\shadow(Y)$, but in general it is not an equality. However, the projection map~$p$ and the decomposition of \ref{notation:decomp} fit into the following commutative diagram:
\begin{equation}
\label{eq:diagram}
\begin{gathered}
\xymatrix@R-1pc@C-1pc{
           & \rips(X_v)\ar@{^{(}->}[rr]\ar@{_{(}->}[dl]\ar^<<<<<p[dd]  &          & \rips(X^v)\ar@{_{(}->}[dl]\ar^p[ddd]\\
\rips(X-v)\ar@{^{(}->}[rr]\ar^p[ddd] &             & \rips(X) \ar^p[ddd]&           \\
           & \shadow(X_v)\ar@{_{(}->}[d]&          &           \\
           & \shadow(X^v)\cap\shadow(X-v)\ar@{^{(}->}[rr] \ar@{_{(}->}[dl] &          & \shadow(X^v)\ar@{_{(}->}[dl]\\
\shadow(X-v)\ar@{^{(}->}[rr] &             & \shadow(X) &          
}
\end{gathered}
\end{equation}
Both the top and the bottom face of the diagram are pushouts in the category of topological spaces (recall that $\rips(X_v)=\rips(X^v)\cap\rips(X-v)$). This puts us in position to study the effects of $p$ inductively. 
For this, we recall some standard homotopy-theoretic notions. Let $k\geq 0$. A continuous map $f\colon X\to Y$ is called \emph{$k$-connected} if the induced map $f_\ast\colon\pi_i(X,x)\to\pi_i(Y,f(x))$ is a bijection for all $0\leq i< k$ and a surjection for~$i=k$, for all choices of basepoint $x\in X$.\footnote{The authors of \cite{chambers2010vietoris} use the term $k$-connected to describe the situation when the induced map on $\pi_k$ is also a bijection, although it is more standard to call the latter a \emph{$k$-equivalence}.}
A \emph{cellular triad} is a triple $(X;A,B)$ of CW-complexes such that $X=A\cup B$. A map $f\colon (X;A,B)\to(Y;C,D)$ of cellular triads is a map $f\colon X\to Y$ such that $f(A)\subseteq C$ and $f(B)\subseteq D$. 

\begin{proposition}[\protect{\cite[Thm. 6.7.9]{tomDieck}}]
\label{prop:mayer}
Let $k\geq 1$. Suppose $f\colon (X;A,B)\to(Y;C,D)$ is a map of cellular triads. If $f\colon A\to C$ and $f\colon B\to D$ are $k$-connected and $f\colon A\cap B\to C\cap D$ is $(k-1)$-connected then $f\colon X\to Y$ is $k$-connected.
\end{proposition}

\begin{figure}
\label{fig:Xv-example}
\includegraphics[scale=1]{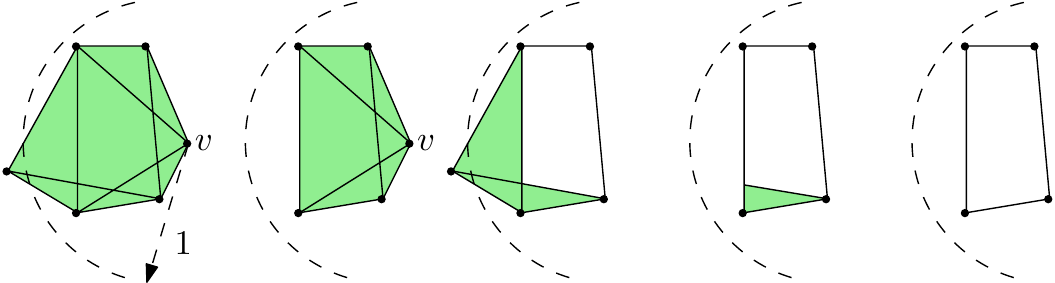}
\caption{From left to right: $\shadow(X)$, $\shadow(X^v)$, $\shadow(X-v)$,  $\shadow(X^v)\cap\shadow(X-v)$ and $\shadow(X_v)$ for a six-point set $X\subseteq \RR^2$.}
\end{figure}

The diagram \eqref{eq:diagram} determines a map of cellular triads $$p\colon(\rips(X);\rips(X-v),\rips(X^v))\to(\shadow(X);\shadow(X-v),\shadow(X^v)).$$ 
An inductive argument using Proposition~\ref{prop:mayer} reduces the question of $1$-connectivity of $p$ to the question of $0$-connectivity of the inclusion map $\shadow(X_v)\incl \shadow(X^v)\cap \shadow(X-v)$. We will study this map in Section~\ref{sect:local}. The geometric input for its analysis is provided in the next section.

\section{Some 3D Euclidean geometry}
\label{sect:geometry}

The arguments in \cite{chambers2010vietoris} rely to a large extent on the following property of the plane.

\begin{proposition}[\protect{\cite[Prop. 2.1]{chambers2010vietoris}}]
\label{prop:tlemma}
Suppose $A,B,P,Q\in\RR^2$ are points such that $|AB|<1$, $|PQ|<1$ and the segments $\conv\{A,B\}$ and $\conv\{P,Q\}$ intersect. Then one of the points $A,B,P,Q$ is in distance less than $1$ from all the remaining ones.
\end{proposition}

The proof is an exercise in manipulations with the triangle inequality. Our next key proposition extends this observation to configurations in~$\RR^3$. The proof, however, is not a simple extension of the proof in~\cite{chambers2010vietoris}, since there are metrics in $\RR^3$ for which the result does not hold, and so the triangle inequality itself is not sufficient (Remark~\ref{rem:funky-metric}). At the same time the result subsumes~\cite[Prop. 2.1 and 2.2]{chambers2010vietoris} for the Euclidean metric in~$\RR^2$.

\begin{proposition}
\label{prop:beer}
Suppose $A,B,C,P,Q$ are points in $\RR^3$ (not necessarily distinct), such that
$|AB|,|BC|,|CA|,|PQ|<1$. Suppose further that $\conv\{P,Q\}\cap\conv\{A,B,C\}\neq\emptyset$. Then one of the points $A,B,C,P,Q$ is in distance less than $1$ from all the remaining ones.
\end{proposition}

\begin{definition}
\label{def:apex}
We call any point which satisfies the conclusion of the proposition an \emph{apex} of the configuration $\{A,B,C,P,Q\}$.
\end{definition}

\begin{proof}
If $P$ (or $Q$) coincides with one of the points $A,B,C$ then it is an apex and we are done. Otherwise we start with the following claim.
\begin{claim}
If $|PB|,|PC|\geq 1$ then $|QA|<1$.
\end{claim}
\begin{proof}
Suppose, on the contrary, that $|QA|\geq 1$. Then we have
\begin{equation}
\label{eq:bisector}
|PQ|<|AQ|,\ |AB|<|PB|,\ |AC|<|PC|.
\end{equation}
Let $AP^\perp$ denote the locus of points $X$ such that $|AX|=|PX|$, that is the $2$-dimensional affine plane which perpendicularly bisects the segment~$AP$. Conditions \eqref{eq:bisector} imply that $Q$ lies (strictly) on the same side of $AP^\perp$ as $P$, while $B$ and $C$ lie on the same side of $AP^\perp$ as~$A$. This contradicts the assumption that $\conv\{P,Q\}\cap\conv\{A,B,C\}\neq\emptyset$.
\end{proof}
We can now complete the proof of the proposition by simple case analysis. If none of the points $A,B,C$ is an apex then by the pigeonhole principle, without loss of generality, we may assume $|PB|,|PC|\geq 1$, and then $|QA|<1$ by the claim. It follows that $|PA|\geq 1$, since otherwise we could choose $A$ as an apex. Now $|PA|,|PB|,|PC|\geq1$ and applying the claim three times we get $|QA|,|QB|,|QC|< 1$, hence $Q$ is an apex.
\end{proof}

\begin{remark}
\label{rem:convex-section}
The proof of Proposition~\ref{prop:beer} uses only the following property of the Euclidean distance: for any two distinct points $S,T\in\RR^n$ the set $\{X\in\RR^n~:~|XS|<|XT|\}$ is convex. 
\end{remark}

\begin{remark}
\label{rem:funky-metric}
Proposition~\ref{prop:beer} fails for some metrics in $\RR^3$. For instance, take the product metric in $\RR^3=\RR^1\times\RR^2$ given by $d((x_1,y_1,z_1),(x_2,y_2,z_2))=|x_1-x_2|+\sqrt{(y_1-y_2)^2+(z_1-z_2)^2}$. Let a unit side equilateral triangle $ABC$ in the plane $x=0$ (the restriction of $d$ to this plane is Euclidean) intersect an orthogonal segment $PQ$ of unit length so that their barycenters coincide. Then $|PA|=\cdots=|QC|=\frac12+\frac{1}{\sqrt{3}}>1$. Rescaling slightly we can assure $|AB|,|BC|,|CA|,|PQ|<1$ with all the other pairwise distances still greater than~$1$.
\end{remark}

The next proposition states, intuitively speaking, that if two tetrahedra in $\RR^3$ have nonempty intersection, and neither is contained in the other, then some facet of one of them intersects some edge of the other. This is an analogue of the fact that if two triangles in $\RR^2$ intersect then they have a pair of intersecting sides (or one is contained in the other). The more involved technical formulation is due to the fact that the two intersecting tetrahedra are allowed to be degenerate.

\begin{proposition}
\label{prop:tetrahedra}
Suppose $\nicea,\niceb\subseteq\RR^3$ are sets with $|\nicea|,|\niceb|\leq 4$ and  $\conv\nicea\cap\conv\niceb\neq\emptyset$. Then at least one of the following holds:
\begin{itemize}
\item[(a)] $\conv\nicea\subseteq\conv\niceb$ or $\conv\niceb\subseteq\conv\nicea$,
\item[(b)]  for some $A,B,C\in \nicea$ and $P,Q\in\niceb$ we have $\conv\{A,B,C\}\cap\conv\{P,Q\}\neq\emptyset$.
\item[(c)]  for some $P,Q\in \nicea$ and $A,B,C\in\niceb$ we have $\conv\{P,Q\}\cap\conv\{A,B,C\}\neq\emptyset$.
\end{itemize}
The choices for $A,B,C$ as well as for $P,Q$ need not be pairwise distinct.
\end{proposition}
\begin{proof}
If neither $\conv\nicea$,  $\conv\niceb$ contains the other then their boundaries intersect, i.e. we can find $\nicea'\subseteq\nicea$, $\niceb'\subseteq\niceb$ with $|\nicea'|,|\niceb'|\leq 3$ and $\conv\nicea'\cap\conv\niceb'\neq\emptyset$. We now have two (possibly degenerate) intersecting triangles in $\RR^3$. In that case some edge of one of those triangles intersects the other triangle.
\end{proof}

\section{Local properties of shadows in dimensions up to $3$ and Theorem~\ref{theoremA}}
\label{sect:local}

As indicated in Section~\ref{sect:topology}, the next proposition is the key ingredient in the proof of Theorem~\ref{theoremA}.

\begin{proposition}
\label{prop:pi0epi}
Suppose $X\subseteq \RR^3$ is a finite set and $v\in X$ is any of its points. Then the inclusion map
$$\shadow(X_v)\incl \shadow(X^v)\cap \shadow(X-v)$$
is $0$-connected (that is, a $\pi_0$-epimorphism).
\end{proposition}
\begin{proof}
It suffices to show that for any point $x\in \shadow(X^v)\cap \shadow(X-v)$ there is a continuous path in $\shadow(X^v)\cap \shadow(X-v)$ which starts in $x$ and ends in a point of $\shadow(X_v)$. This will show that every component of $\shadow(X^v)\cap \shadow(X-v)$ contains a component of $\shadow(X_v)$, and prove the proposition.

By the definition of shadows there are subsets $\nicea\subseteq X^v$ and $\niceb\subseteq X-v$ such that 
$$|\nicea|,|\niceb|\leq 4,\quad \diam(\nicea), \diam(\niceb)<1,\quad x\in \conv\nicea\cap\conv\niceb.$$ 

If $\nicea\subseteq X_v$ or $\niceb\subseteq X_v$ then we are done since $x\in \shadow(X_v)$. We can therefore restrict to the case when $v\in \nicea$ and when $\niceb$ contains a point from outside $B(v,1)$. In particular, since $\conv\nicea\subseteq B(v,1)$, that implies $\conv\niceb\not\subseteq\conv\nicea$. If we had $\conv\nicea\subseteq\conv\niceb$ then in particular $v\in\conv\niceb$. Since $\niceb$ has a point outside  $B(v,1)$, we get $1> \diam(\niceb)=\diam(\conv\niceb)\geq 1$, which is a contradiction. It follows that $\conv\nicea\not\subseteq\conv\niceb$.

We conclude that neither of $\conv\nicea$ and $\conv\niceb$ is contained in the other. By Proposition~\ref{prop:tetrahedra} this means that there is a point $y$ in the intersection of a triangle with vertices in $\nicea$ and a segment with vertices in $\niceb$ or vice versa (the triangle and segment can each be degenerate). Since $\conv\nicea\cap\conv\niceb$ is convex, the straight line segment connecting $x$ to $y$ lies entirely in $\conv\nicea\cap\conv\niceb\subseteq\shadow(X^v)\cap\shadow(X-v)$. It remains to find a path connecting $y$ to $\shadow(X_v)$ inside $\shadow(X^v)\cap\shadow(X-v)$. The proof splits into two cases, corresponding to parts (b) and (c) of Proposition~\ref{prop:tetrahedra}. See Figure~\ref{fig:intersections}. In each case $y$ will be connected to $\shadow(X_v)$ by a single segment.

\emph{Case 1.} We have $y\in\conv\{A,B,C\}\cap\conv\{P,Q\}$ for $A,B,C\in\nicea\subseteq X^v$ and $P,Q\in\niceb\subseteq X-v$. Repeating the arguments from the first part of the proof for the new sets $\nicea'=\{A,B,C\}$ and $\niceb'=\{P,Q\}$, we reduce without loss of generality  to the case when $A=v$, $B,C\in X_v$ (possibly $B=C$) and $Q\not\in X_v$. By Proposition~\ref{prop:beer} the configuration $\{A,B,C,P,Q\}$ has an apex. We have two possibilities:
\begin{itemize}
\item If $P$ is an apex then $P\in X_v$. The straight line segment $yP$ belongs to $$\conv\{A,B,C,P\}\cap\conv\{P,Q\}\subseteq \shadow(X^v)\cap\shadow(X-v)$$ and connects $y$ to $P\in X_v$. 
\item Neither $A$ nor $Q$ can be an apex, so the only remaining case, w.l.o.g., is that $B$ is an apex. Then the segment $yB$ is in $\conv\{A,B,C\}\cap\conv\{B,P,Q\}\subseteq \shadow(X^v)\cap\shadow(X-v)$, connecting $y$ to $B\in X_v$.
\end{itemize}

\emph{Case 2.} We have $y\in\conv\{P,Q\}\cap\conv\{A,B,C\}$ for $P,Q\in\nicea\subseteq X^v$ and $A,B,C\in\niceb\subseteq X-v$. As before we can assume that $P=v$, $Q\in X_v$ and that $C\not\in X_v$. Up to relabelling the possibilities are as follows:
\begin{itemize}
\item $A,B\in X_v$. Either $Q$, $A$ or $B$ is an apex, so we must have $|QA|< 1$ or ${|QB|< 1}$. W.l.o.g. suppose that $|QA|< 1$. The segment $yA$ is now in $$\conv\{P,Q,A\}\cap\conv\{A,B,C\}\subseteq \shadow(X^v)\cap\shadow(X-v),$$ connecting $y$ to $A\in X_v$.
\item $A\in X_v$, $B\not\in X_v$. Only $A$ or $Q$ can be an apex, and in either case $|QA|<1$. We conclude as before.
\item $A,B\not\in X_v$. Now $Q$ must be the apex, and the segment $yQ$ belongs to $\conv\{P,Q\}\cap\conv\{A,B,C,Q\}\subseteq \shadow(X^v)\cap\shadow(X-v)$, connecting $y$ to $Q\in X_v$.
\end{itemize}
This ends the proof of the proposition.
\end{proof}

\begin{figure}
\includegraphics[scale=1.2]{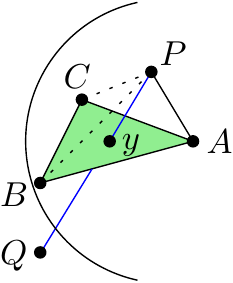} \includegraphics[scale=1.2]{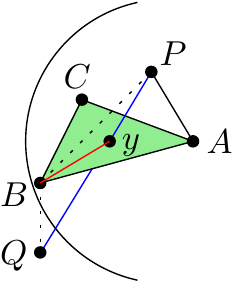} \includegraphics[scale=1.2]{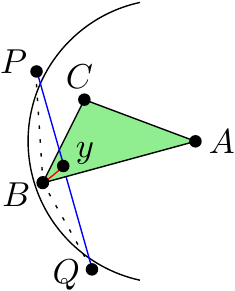}\\
\includegraphics[scale=1.2]{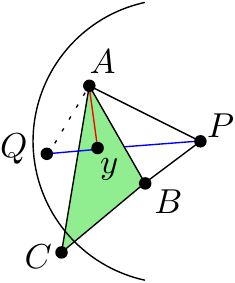} \includegraphics[scale=1.2]{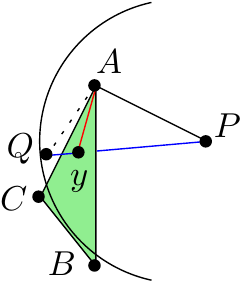} \includegraphics[scale=1.2]{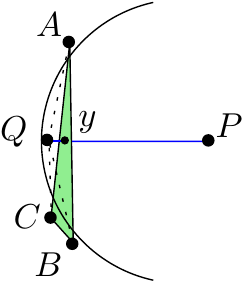}
\caption{Case 1 (top) and case 2 (bottom) in the proof of Proposition~\ref{prop:pi0epi}.}
\label{fig:intersections}
\end{figure}
\begin{figure}
\includegraphics[scale=1.2]{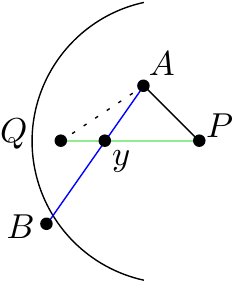} \includegraphics[scale=1.2]{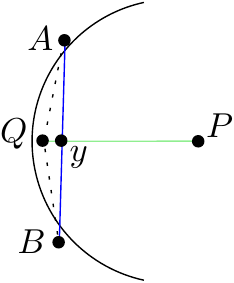}
\caption{The simplified case analysis in Proposition~\ref{prop:pi0epi} for subsets of the plane only.}
\label{fig:intersections2}
\end{figure}

\begin{remark}
\label{rem:dim2easier}
For subsets $X\subseteq \RR^2$ the case analysis is much shorter (Figure~\ref{fig:intersections2}) and requires only Proposition~\ref{prop:tlemma}. Of course the proof above covers the case of $X\subseteq \RR^2\subseteq \RR^3$ anyway.
\end{remark}

\begin{remark}
\label{rem:empty}
The proposition holds also when $X_v=\emptyset$, when it implies that $\shadow(X^v)\cap{\shadow(X-v)}=\emptyset$. This is easily verified in any dimension with an argument from the first part of the proof.
\end{remark}

\begin{remark}
\label{rem:dim4fail}
Proposition~\ref{prop:pi0epi} fails in dimension~$4$. The counterexample in \cite[Prop. 5.4]{chambers2010vietoris} works to show that and we recall it briefly. Let $\xi=\exp(i\pi/3)$ and choose a constant $c$ slightly less than~$1/\sqrt{3}$. The triples $\{c\xi^0,c\xi^2,c\xi^4\}$ and $\{c\xi^1,c\xi^3,c\xi^5\}$ each span a triangle in $\CC$ of diameter less than one. For sufficiently small $\varepsilon$ consider the set $X=\{v_0,\ldots,v_5\}\subseteq\CC^2$ with:
$$v_0=(c\xi^0,0),\ v_2=(c\xi^2,0),\ v_4=(c\xi^4,0),\ v_1=(c\xi^1,\varepsilon\xi^1),\ v_3=(c\xi^3,\varepsilon\xi^3),\ v_5=(c\xi^5,\varepsilon\xi^5).$$
Note that $|v_iv_{i+3}|\approx 2c>1$ and the other distances are less than $1$ for $\varepsilon$ small enough. We see that $\shadow(X_{v_0})$ is the $4$-cycle on $\{v_1,v_2,v_4,v_5\}$. The intersection $\shadow(X^{v_0})\cap\shadow(X-{v_0})$ has two connected components: one is $\shadow(X_{v_0})$ and the other one is the point $(0,0)$ which lies in $\conv\{v_0,v_2,v_4\}\cap\conv\{v_1,v_3,v_5\}$.
\end{remark}

\begin{proof}[Proof of Theorem~\ref{theoremA}]
We can prove by induction that for all finite $X\subseteq \RR^3$ the map $p\colon\rips(X)\to\shadow(X)$ is $1$-connected. Indeed, the statement is true if $X=\emptyset$. Otherwise pick $v\in X$ and observe that in \eqref{eq:diagram} the right vertical map $\rips(X^v)\to\shadow(X^v)$ is $\infty$-connected by Lemma~\ref{lem:stars}, the left map $\rips(X-v)\to\shadow(X-v)$ is $1$-connected by induction, and the map $\rips(X_v)\to\shadow(X_v)\to\shadow(X^v)\cap\shadow(X-v)$ in the back of the diagram is a composition of $0$-connected maps by induction and Proposition~\ref{prop:pi0epi}. Now Proposition~\ref{prop:mayer} implies that the map in the front is $1$-connected. That proves Theorem~\ref{theoremA}.
\end{proof}

\section{Universality}
\label{sect:universal}

The goal of this section is to show that Rips complexes of finite subsets of $\RR^n$ model at least all the homotopy types of finite simplicial complexes embedded in~$\RR^n$. We should remark that this is by no means a complete characterization in dimensions~$n\geq 2$. For instance for every $k\geq 0$ there is a finite subset $X\subseteq \RR^2$ for which $\rips(X)$ is homeomorphic to the sphere~$S^k$.

Suppose we have an abstract simplicial complex $K$ with a fixed simplex-wise linear embedding in some~$\RR^n$. 
We will denote the image of this embedding by $|K|\subseteq \RR^n$. We would like to take a dense subset $X\subseteq |K|$ and a complex $\rips(X;\varepsilon)$ for $\varepsilon>0$ small enough so that the points from disjoint faces of $K$ do not see each other at scale $\varepsilon$. 

Let us write $\ost_K(\sigma)\subseteq |K|$ for the open star of $\sigma$ with respect to the given embedding. It is defined as the union of the interiors $|\mathrm{Int}(\tau)|$ over all faces $\tau$ containing~$\sigma$. It is a standard fact that the collection $\{\ost_K(v)\}_{v\in V(K)}$ forms a cover of $|K|$ by open sets, with all intersections either empty or contractible and with nerve isomorphic to~$K$. Indeed, 
\begin{equation}
\label{eq:nerve}
\ost_K(v_1)\cap\cdots\cap\ost_K(v_t)=
\begin{cases}
\emptyset & \mathrm{if}\  \{v_1,\ldots,v_t\}\not\in K\\
\ost_K(\sigma) & \mathrm{if}\  \{v_1,\ldots,v_t\}=:\sigma\in K,
\end{cases}
\end{equation}
and each set $\ost_K(\sigma)$ is contractible.

To simplify the exposition we will first show that an embedded simplicial complex is homotopy equivalent to its thickenings obtained as Rips complexes on \emph{all} the points of $|K|$ at sufficiently small distance scales. Here is where a Rips complex with an infinite vertex set appears. The proof uses a covering induced by the one above and the Nerve Lemma. The next lemma will ensure contractibility of the pieces of the cover.

\begin{lemma}
\label{lem:crushing}
If $F\subseteq \RR^n$ is a bounded star-shaped set then for any $r>0$ the Rips complex $\rips(F;r)$ is contractible.
\end{lemma}
\begin{proof}
Let $o$ be the center of $F$ and consider the linear homotopy $h:F\times [0,1]\to F$ given by $h(x,t)=tx+(1-t)o$. It scales distances, i.e. $|h(x,t)h(y,t)|=t|xy|$, and in particular $|h(x,t)h(y,t)|\geq|h(x,t')h(y,t')|$ for $t\geq t'$, which means that $h$ is a \emph{crushing}  of $F$ (to the point~$o$) in the sense of Hausmann \cite[p.177]{Hausmann1995}. By \cite[(2.3)]{Hausmann1995} the Rips complex, at any scale, of a crushable metric space is contractible.
\end{proof}

\begin{theorem}
\label{thm:infinite}
For every finite simplicial complex $K$ with a linear embedding $|K|\subseteq\RR^n$ there is an $\varepsilon_0>0$, such that for all $0<\varepsilon<\varepsilon_0$ we have a homotopy equivalence $\rips(|K|;\varepsilon)\htpyequiv K$.
\end{theorem}

\begin{proof}
We will first verify two claims which hold independently of~$\varepsilon$.

\begin{customclaim}{A}
The nerve of the collection of simplicial complexes $\{\rips(\ost_K(v);\varepsilon)\}_{v\in V(K)}$ is isomorphic to $K$ for any~$\varepsilon>0$.
\end{customclaim}
\begin{proof}
The nerve of a collection of abstract simplicial complexes is isomorphic to the nerve of their vertex sets, hence the statement follows from~\eqref{eq:nerve}.
\end{proof}

\begin{customclaim}{B}
Each complex $\rips(\ost_K(\sigma);\varepsilon)$ is contractible for any $\emptyset\neq \sigma\in K$ and $\varepsilon>0$.
\end{customclaim}
\begin{proof}
Since $|K|\subseteq \RR^n$ is a linear embedding of $K$, the set $\ost_K(\sigma)\subseteq \RR^n$ is a star-shaped subset with center in the barycenter of~$|\sigma|$. The claim now follows from Lemma~\ref{lem:crushing}. 
\end{proof}

Finally, we need to ensure that the collection of subcomplexes $\{\rips(\ost_K(v);\varepsilon)\}_{v\in V(K)}$ covers~$\rips(|K|;\varepsilon)$. Here is where the choice of $\varepsilon$ matters. Define
\begin{equation}
\label{eq:epsilon0}
\varepsilon_0=\inf_{\substack{\sigma_1,\ldots,\sigma_t\in K,\ \sigma_i\neq \emptyset\\ \sigma_1\cap\cdots\cap\sigma_t=\emptyset}}\inf\{
\diam(\{x_1,\ldots,x_t\})~:~x_i\in|\sigma_i|\ \mathrm{for\ all}\ i\}.
\end{equation}
The diameter of a set is not smaller than the diameter of any of its subsets. It follows that for the first infimum we can restrict to the configurations with $\sigma_i\neq\sigma_j$ for all $1\leq i<j\leq t$. Since $K$ is finite, there are only finitely many of those.

Since $|K|$ is an embedding of $K$ the condition $\bigcap_{i=1}^t\sigma_i=\emptyset$ implies $\bigcap_{i=1}^t|\sigma_i|=\emptyset$. It follows that the continuous function $\prod_{i=1}^t|\sigma_i|\to\RR$ defined as $(x_1,\ldots,x_t)\to\diam\{x_1,\ldots,x_t\}$ does not hit $0$. Since $\prod_{i=1}^t|\sigma_i|$ is compact, the function attains a strictly positive minimum. Together with the previous remark this implies~$\varepsilon_0>0$. Now we come to the last claim.

\begin{customclaim}{C}
If $0<\varepsilon<\varepsilon_0$ then the collection $\{\rips(\ost_K(v);\varepsilon)\}_{v\in V(K)}$ is a covering of $\rips(|K|;\varepsilon)$.
\end{customclaim}
\begin{proof}
Suppose that $\{x_1,\ldots,x_t\}$ is a face of $\rips(|K|;\varepsilon)$. For $i=1,\ldots,t$ let $\sigma_i$ be the smallest-dimensional face of $K$ such that $x_i\in |\sigma_i|$. Since $\diam\{x_1,\ldots,x_t\}<\varepsilon<\varepsilon_0$, by definition of $\varepsilon_0$ we conclude $\bigcap_{i=1}^t\sigma_i\neq\emptyset$. Choose any vertex $v\in\bigcap_{i=1}^t\sigma_i$. Then each $x_i$ lies in $\ost_K(v)$, and consequently $\{x_1,\ldots,x_t\}$ is a face of $\rips(\ost_K(v);\varepsilon)$.
\end{proof}
The three claims imply that for $0<\varepsilon<\varepsilon_0$ the collection $\{\rips(\ost_K(v);\varepsilon)\}_{v\in V(K)}$ is a good cover of $\rips(|K|;\varepsilon)$ with nerve isomorphic to $K$. By the Nerve Lemma (for subcomplexes of a simplicial complex \cite[Theorem 10.6]{Bjorner1995}) that completes the proof.
\end{proof}

Our next step is to replace the infinite vertex set $|K|$ of $\rips(|K|;\varepsilon)$ with a sufficiently dense finite sample. Given $X\subseteq F\subseteq \RR^n$ and $\delta>0$ we say that \emph{$X$ is $\delta$-dense in $F$} if for every $p\in F$ there is some $q\in X$ such that $|pq|<\delta$. The next lemma is a version of Lemma~\ref{lem:crushing} with finite samples.

\begin{lemma}
\label{lem:crulat}
For every bounded star-shaped set $F\subseteq \RR^n$ and every $\varepsilon>0$ there exists a $\delta>0$ with the following property. If $X\subseteq F$ is finite and $\delta$-dense in $F$ then $\rips(X;\varepsilon)$ is contractible.
\end{lemma}
\begin{proof}
This was shown by Latschev \cite[Lemma 2.1, Remark 2.2]{Latschev2001} in the case when $F$ is convex, but a careful inspection of the proof reveals that it applies verbatim when $F$ is star-shaped. Since this lemma is key for the proof of Theorem~\ref{theoremB}, we will present the proof for completeness.

Without loss of generality assume that $0\in F$ is the center of $F$ and write $\lVert p\rVert=|0p|$. Choose $R$ so that $F\subseteq B(0,R)$. For every $p\in F$ with $\lVert p\rVert\geq\frac12\varepsilon$ let $c(p)$ be the center of the sphere $\partial B(0,\lVert p\rVert)\cap\partial B(p,\varepsilon)$. We easily compute
$c(p)=p\cdot(\lVert p\rVert-\frac{\varepsilon^2}{2\lVert p\rVert})$ and $|c(p)p|=\frac{\varepsilon^2}{2\lVert p\rVert}$. The radius of the sphere $\partial B(0,\lVert p\rVert)\cap\partial B(p,\varepsilon)$ is $\varepsilon\sqrt{1-\frac{\varepsilon^2}{4\lVert p\rVert^2}}$.

Define $\delta=\varepsilon-\varepsilon\sqrt{1-\frac{\varepsilon^2}{4R^2}}$ and suppose $X$ is $\delta$-dense in $F$. Since $F$ is star-shaped, $p \in F$ implies $c(p)\in F$, and we can find a point $g(p)\in X$ with $|g(p)c(p)|<\delta$. The above geometric considerations and the triangle inequality now give:
\begin{equation}
\label{eq:balls}
\ \overline{B}(0,\lVert p\rVert)\cap \overline{B}(p,\varepsilon)\subseteq  \overline{B}\Big( c(p),\varepsilon\sqrt{1-\frac{\varepsilon^2}{4\lVert p\rVert^2}}\Big)\subseteq  B\Big(c(p),\varepsilon\sqrt{1-\frac{\varepsilon^2}{4R^2}}\Big)\subseteq B(g(p),\varepsilon).
\end{equation}
Moreover, we check that for $\lVert p\rVert\geq \frac12\varepsilon$ we have
$$ \delta<\varepsilon-\varepsilon\Big(1-\frac{\varepsilon^2}{4R^2}\Big)=\frac{\varepsilon^3}{4R^2}<\frac{\varepsilon^3}{4\lVert p\rVert^2}=|c(p)p|\cdot \frac{\varepsilon}{2\lVert p\rVert}\leq |c(p)p|$$
which implies a strict inequality $\lVert g(p)\rVert<\lVert p\rVert$.

Order the elements of $X$ as $X=\{x_1,\ldots,x_s\}$ with $\lVert x_1\rVert\leq \cdots\leq \lVert x_s\rVert$ and let $X_i=\{x_1,\ldots,x_i\}$ for $i=1,\ldots,s$. We will show by induction that $\rips(X_i;\varepsilon)$ is contractible, and for $i=s$ that proves the lemma. Of course $\rips(X_1;\varepsilon)$ is contractible.

If $\lVert x_i\rVert<\frac12\varepsilon$ then $X_i\subseteq B(0,\frac12\varepsilon)$ and $\rips(X_i;\varepsilon)$ is a simplex. Suppose $\lVert x_i\rVert\geq\frac12\varepsilon$. Every point of $X_i$ that is $\varepsilon$-close to $x_i$ lies in $\overline{B}(0,\lVert x_i\rVert)\cap B(x_i,\varepsilon)$, and by \eqref{eq:balls} it is $\varepsilon$-close to~$g(x_i)$. Since $g(x_i)\in X_{i-1}$, we conclude that the link of $x_i$ in $\rips(X_i;\varepsilon)$ is a cone with apex in $g(x_i)$. As the link of $x_i$ is contractible, we have $\rips(X_i;\varepsilon)\htpyequiv \rips(X_{i-1};\varepsilon)$, completing the inductive step.
\end{proof}

The next theorem implies Theorem~\ref{theoremB}. In the proof we will recycle parts of the proof of Theorem~\ref{thm:infinite}.

\begin{theorem}[Theorem~\ref{theoremB}]
\label{thm:finite}
For every finite simplicial complex $K$ with a linear embedding $|K|\subseteq\RR^n$ there is an $\varepsilon>0$ and a finite set $X\subseteq |K|$, such that we have a homotopy equivalence $\rips(X;\varepsilon)\htpyequiv K$.
\end{theorem}
\begin{proof}
Define $\varepsilon_0$ as in the proof of Theorem~\ref{thm:infinite} and fix any $0<\varepsilon<\varepsilon_0$. With this $\varepsilon$, pick $\delta>0$ small enough to satisfy Lemma~\ref{lem:crulat} simultaneously for all star-shaped sets $\ost_K(\sigma)$, where $\emptyset\neq\sigma\in K$. Let $X\subseteq |K|$ be any finite set such that $X\cap \ost_K(\sigma)$ is $\delta$-dense in $\ost_K(\sigma)$ for all $\emptyset\neq\sigma\in K$.

We now have the following analogues of the claims from the proof of Theorem~\ref{thm:infinite} restricted to the points of~$X$.
\begin{itemize}
\item[(A')] The nerve of the collection of simplicial complexes $\{\rips(X\cap \ost_K(v);\varepsilon)\}_{v\in V(K)}$ is isomorphic to~$K$.
\item[(B')] Each simplicial complex $\rips(X\cap \ost_K(\sigma);\varepsilon)$ is contractible for any $\emptyset\neq \sigma\in K$.
\item[(C')] The collection $\{\rips(X\cap \ost_K(v);\varepsilon)\}_{v\in V(K)}$ is a covering of~$\rips(X;\varepsilon)$.
\end{itemize}
The proofs are identical as in Theorem~\ref{thm:infinite}, with two modifications. For (A') we need to use the fact that $X\cap \ost_K(\sigma)$ is nonempty for all $\emptyset\neq\sigma\in K$. For (B') we apply Lemma~\ref{lem:crulat} instead of \ref{lem:crushing}.
Using the Nerve Lemma we conclude that $\rips(X;\varepsilon)\htpyequiv K$ as before.
\end{proof}

As a corollary we obtain Theorem~\ref{theoremC}.

\begin{proof}[Proof of Theorem~\ref{theoremC}]
For every finitely presented group $G$ there exists a $2$-dimensional polyhedron $M\subseteq \RR^4$ with $\pi_1(M)=G$ \cite{dranivsnikov1993}. Since a polyhedron can be triangulated, we can assume that $M=|K|$ for some simplicial complex $K$ linearly embedded in~$\RR^4$. Theorem~\ref{theoremB} gives a finite set $X\subseteq \RR^4$ with $\rips(X)\htpyequiv K$. In particular $\pi_1(\rips(X))=\pi_1(K)=\pi_1(M)=G$.
\end{proof}

\section{Submanifolds in planar Rips complexes}
\label{sect:submanifold}

Here we show that if $X\subseteq \RR^2$ is a finite set and $\rips(X)$ is a normal pseudomanifold of dimension $d\geq 2$ then it must be the boundary of a crosspolytope. This is motivated by the question \cite[6.(1)]{chambers2010vietoris} about the structure of generators of $H_d(\rips(X))$, and by the attempts to verify that $\rips(X)$ is homotopy equivalent to a wedge of spheres (Problem~\ref{problem:a}).

A \emph{$d$-dimensional normal pseudomanifold} is a pure, connected simplicial complex $K$ of dimension $d$ such that every $(d-1)$-face is contained in precisely two $d$-faces and the link of every face of dimension at most $d-2$ is connected. It is easy to show that for any vertex $v\in V(K)$ the link $\lk_K(v)$ is a $(d-1)$-dimensional normal pseudomanifold. Any $2$-dimensional normal pseudomanifold is a surface without boundary.

The first, yet most important step, is to classify planar Rips complexes which triangulate $2$-dimensional closed surfaces. Note that if $\rips(X)$ triangulates a surface, every edge of $\rips(X)$ belongs to exactly two triangles, and for every $v\in X$ the link $\rips(X_v)$ is a cycle with at least $4$ vertices. The shadow $\shadow(X_v)\subseteq B(v,1)$ is a simple polygon, as self-intersections would yield additional edges by Proposition~\ref{prop:tlemma}. In particular, we can speak of points inside, resp. outside the polygon $\shadow(X_v)$ (i.e. in the bounded, resp. unbounded region of the plane determined by the polygon).

For clarity we first formulate an elementary lemma.
\begin{lemma}
\label{lem:help-help}
Let $\nices$ be a simple polygon in $\RR^2$ and $v$ a point outside or on $\nices$. Suppose further that for every vertex $P$ of $\nices$ all the intersections of the ray $vP^\to$ with $\nices$ lie between $v$ and $P$. 

Then there is an edge $AB$ of $\nices$ such that for any vertex $P$ of $\nices$ the segments $vP$ and $AB$ have non-empty intersection.
\end{lemma}

\begin{proof}
	This is trivial for $v$ on~$\nices$. Otherwise consider the set $C$ of points on~$\nices$ visible from~$v$. Let $\Omega$ be the
	set of angles $\omega$ such that the ray with angle $\omega$ starting at $v$ intersects~$C$. The map  $\Omega \to C$ that 
	maps an angle to the corresponding intersection point is continuous, since a point of discontinuity would be a vertex that 
	contradicts the prerequisites. Thus $C$ is connected. Similar reasoning shows that if $C$ contains part of some open edge
	it must contain the entire edge. Moreover if $C$ contained a vertex $P$ in its relative interior then the ray $vP^\to$ would
	cross into the region bounded by~$\nices$ at $P$ and thus contradict the prerequisites. Thus $C$ consists of a single edge~$AB$.
\end{proof}

Now we come to the main proposition. The proof does not depend on the properties of $\pi_1p$ from Theorem~\ref{thm:chambers} or Theorem~\ref{theoremA}.

\begin{proposition}
\label{prop:surface-cross}
Suppose that $X\subseteq \RR^2$ is a finite set such that $\rips(X)$ is homeomorphic to a surface (without boundary). Then $|X|=6$, and $\rips(X)$ is isomorphic to the boundary complex of the octahedron.
\end{proposition}
\begin{proof}
We will show that $|X_v|=4$ for any $v\in X$. If this is true, then $\rips(X)$ has $2|X|$ edges, while a triangulated surface has at least $3|X|-6$ edges, yielding $|X|\leq 6$. The only flag surface triangulation with at most $6$ vertices is the boundary of the octahedron, and that observation completes the proof.

Pick any $v\in X$. Suppose that $v$ lies strictly inside the polygon $\shadow(X_v)$. We first show that $v\not\in\shadow(X-v)$. Otherwise, there would be points $A,B,C\in X-v$ such that $\diam\{A,B,C\}<1$ and $v\in \conv\{A,B,C\}$. Since $\shadow(X_v)$ is just a cycle, not passing through $v$, not all of $A,B,C$ can be in $X_v$, say $A\not\in X_v$. But then $1>\diam\{A,B,C\}=\diam(\conv\{A,B,C\})\geq |Av|\geq 1$, a contradiction. Since $\rips(X)$ is a triangulated surface, the inclusion $\rips(X_v)\incl \rips(X-v)$ is null-homologous, hence so is the composition
$$\rips(X_v)\incl \rips(X-v)\xrightarrow{p} \shadow(X-v)\incl \RR^2- v.$$
However, this map is equal to the composition
$$\rips(X_v)\xrightarrow{p}\shadow(X_v)\incl\RR^2-v$$
of a homeomorphism followed by a homotopy equivalence (by the assumption that $v$ lies inside the polygon $\shadow(X_v)$). This is a contradiction, which shows that in fact $v$ must lie outside (or on some edge of) the simple polygon $\shadow(X_v)$.

Suppose that there is a point $P\in X_v$ such the ray $vP^{\to}$ intersects some edge $AB$ of $\shadow(X_v)$ with $A,B\neq P$, and so that $v$, $P$, and the intersection point lie on $vP^\to$ in this order (Figure~\ref{fig:simplepolygons}, left). Then $P\in\conv\{v,A,B\}$, therefore $vPAB$ is a simplex in $\rips(X)$, contradicting the fact that $\rips(X)$ is $2$-dimensional.

\begin{figure}
\includegraphics{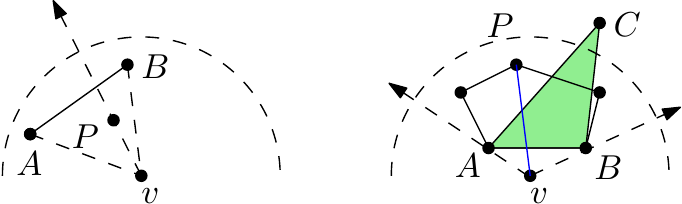}
\caption{The possible layouts of $\shadow(X_v)$ in the proof of Proposition~\ref{prop:surface-cross}.}
\label{fig:simplepolygons}
\end{figure}

We conclude that for any $P\in X_v$ the part of the ray $vP^\to$ which extends beyond $P$ does not intersect $\shadow(X_v)$. Since $v$ lies outside (or on) $\shadow(X_v)$, we are in the situation of Lemma~\ref{lem:help-help}. We conclude that the only part of $\shadow(X_v)$ visible from $v$ is a single edge $AB$ (Figure~\ref{fig:simplepolygons}, right) and for all $P\in X_v$ we have $\conv\{v,P\}\cap\conv\{A,B\}\neq\emptyset$.

To obtain the final contradiction suppose that $|X_v|\geq 5$. Then there is a point $P\in X_v$ with $|PA|,|PB|\geq 1$. The edge $AB$ belongs to two faces of $\rips(X)$: one of them is $ABv$ and the other is $ABC$ for some $C\neq A,B,v$, necessarily with $C\not\in X_v$. As noted above, the segment $\conv\{v,P\}$ and the triangle $\conv\{A,B,C\}$ intersect along $\conv\{A,B\}$. The combination of constraints
$$|PA|\geq 1,\ |PB|\geq 1,\ |vC|\geq 1,\ \diam\{A,B,C\}<1,\ \diam\{P,v\}<1$$
means that the configuration $\{A,B,C,P,v\}$ has no apex in the sense of Definition~\ref{def:apex}, although it satisfies the assumptions of Proposition~\ref{prop:beer} (or even its planar version \cite[Prop.2.2]{chambers2010vietoris}). This contradiction implies $|X_v|\leq 4$, hence in fact $|X_v|=4$, and the proof is complete.
\end{proof}

We can inductively classify normal pseudomanifolds which can appear as an \emph{induced} subcomplex of a planar Rips complex. Let $\niceo_d$ denote the boundary of the $(d+1)$-dimensional crosspolytope. For example, $\niceo_1$ is the cycle on $4$ vertices, $\niceo_2$ is the boundary of the octahedron, and in general $\niceo_d$ is the suspension of~$\niceo_{d-1}$.

\begin{theorem}[Theorem~\ref{theoremD}]
\label{thm:planar-alldim}
For any $d\geq 2$ and any finite subset $X\subseteq \RR^2$, if $\rips(X)$ is a $d$-dimensional normal pseudomanifold then $\rips(X)$ is isomorphic to~$\niceo_d$.
\end{theorem}
\begin{proof}
A $2$-dimensional normal pseudomanifold triangulates a surface, hence the case $d=2$ follows from Proposition~\ref{prop:surface-cross}. Let $d\geq 3$. All vertex links $\rips(X_v)$ are $(d-1)$-dimensional normal pseudomanifolds, and by induction they are all isomorphic to $\niceo_{d-1}$. It is easy to show that the only complex in which all vertex links are isomorphic to $\niceo_{d-1}$ is $\niceo_d$; for example, see~\cite[Theorem~2.1]{deza}.
\end{proof}

\begin{example}
\label{ex:crosspolytope}
Of course every crosspolytopal sphere can be realized in this fashion: we have $\niceo_d=\rips(X_{2(d+1)})$ for $X_{2(d+1)}$ the vertex set of a regular $2(d+1)$-gon inscribed in a circle of radius $\frac12$. For $d=2$ an isomorphic $6$-point configuration appears in Figure~\ref{fig:Xv-example}.
\end{example}

\section{Closing remarks}
\label{sect:closing}
We close the paper with some conjectures. The first one is a strengthening of Proposition~\ref{prop:pi0epi}.
\begin{conjecture}
\label{conj:pi1-3}
For any finite subset $X\subseteq \RR^3$ and any $v\in X$ the inclusion map $\shadow(X_v)\incl \shadow(X^v)\cap\shadow(X-v)$ is $1$-connected.
\end{conjecture}
Using the method of Section~\ref{sect:topology} this conjecture would imply that for finite $X\subseteq\RR^3$ the projection $p\colon\rips(X)\to\shadow(X)$ is $2$-connected, hence a $\pi_1$-isomorphism (and even a $\pi_2$-surjection, a statement which is vacuously true for $X\subseteq \RR^2$). 

We also believe that $\pi_1$-injectivity holds in general.
\begin{conjecture}
\label{conj:pi1-inj}
The projection $p$ induces a $\pi_1$-injection in all dimensions.
\end{conjecture}

Finally, this is perhaps a good point to mention the following question that has been around for some time now.
\begin{problem}
\label{problem:a}
Can we characterize the homotopy types of Rips complexes $\rips(X)$ for $X\subseteq\RR^2$? In particular, are they always homotopy equivalent to wedges of spheres?
\end{problem}

Theorem~\ref{thm:planar-alldim} might be taken as some evidence in favor of an affirmative answer to this question.

\bibliographystyle{plain}
\bibliography{rips}

\begin{thebibliography}{10}

\bibitem{Bjorner1995}
Anders Bj{\"o}rner.
\newblock Topological methods.
\newblock {\em Handbook of Combinatorics}, 2:1819--1872, 1995.

\bibitem{chambers2010vietoris}
Erin~W Chambers, Vin de~Silva, Jeff Erickson, and Robert Ghrist.
\newblock {V}ietoris--{R}ips complexes of planar point sets.
\newblock {\em Discrete \& Computational Geometry}, 44(1):75--90, 2010.

\bibitem{chazal2014persistence}
Fr{\'e}d{\'e}ric Chazal, Vin De~Silva, and Steve Oudot.
\newblock Persistence stability for geometric complexes.
\newblock {\em Geometriae Dedicata}, 173(1):193--214, 2014.

\bibitem{deza}
Michel Deza, Mathieu Dutour, and Mikhail Shtogrin.
\newblock On simplicial and cubical complexes with short links.
\newblock {\em Israel Journal of Mathematics}, 144(1):109--124, 2004.

\bibitem{dranivsnikov1993}
Aleksandr~N. Drani{\v{s}}nikov and Du{\v{s}}an Repov{\v{s}}.
\newblock {Embedding up to homotopy type in Euclidean space}.
\newblock {\em Bull. Austral. Math. Soc.}, 47(1):145--148, 1993.

\bibitem{Hausmann1995}
Jean-Claude Hausmann.
\newblock On the {V}ietoris--{R}ips complexes and a cohomology theory for
  metric spaces.
\newblock {\em Annals of Mathematics Studies}, 138:175--188, 1995.

\bibitem{Kozlov}
Dmitry~N Kozlov.
\newblock {\em Combinatorial Algebraic Topology}, volume~21 of {\em Algorithms
  and Computation in Mathematics}.
\newblock Springer, 2008.

\bibitem{Latschev2001}
Janko Latschev.
\newblock Vietoris--{R}ips complexes of metric spaces near a closed
  {R}iemannian manifold.
\newblock {\em Archiv der Mathematik}, 77(6):522--528, 2001.

\bibitem{tomDieck}
Tammo tom Dieck.
\newblock {\em Algebraic Topology}.
\newblock EMS {T}extbooks in {M}athematics. European Mathematical Society,
  2008.

\bibitem{Vietoris27}
Leopold Vietoris.
\newblock {\"Uber den h\"oheren Zusammenhang kompakter R\"aume und eine Klasse
  von zusammenhangstreuen Abbildungen}.
\newblock {\em Mathematische Annalen}, 97(1):454--472, 1927.

\end{thebibliography}

\end{document}